\documentclass[11pt,a4paper,reqno]{amsart}
\usepackage{pdfsync}
\usepackage{mathptmx} 
\usepackage[scaled=0.90]{helvet} 
\usepackage{courier}
\normalfont
\usepackage[T1]{fontenc}

\usepackage[applemac]{inputenc}
\usepackage[T1]{fontenc}

\usepackage{microtype}  
\usepackage[shortcuts]{extdash}  
  
\usepackage{amssymb,amsrefs}
\usepackage{marvosym}
\usepackage{graphicx}
\usepackage{epstopdf}
\usepackage{lscape}
\usepackage{color}
\usepackage{verbatim}

\renewcommand{\epsilon}{\varepsilon}

\renewcommand{\emptyset}{\varnothing}

\newtheorem{theorem}{Theorem}[section]
\newtheorem{proposition}[theorem]{Proposition}
\newtheorem{corollary}[theorem]{Corollary}
\newtheorem{lemma}[theorem]{Lemma}

\newtheorem*{Main Theorem}{Main Theorem}

\theoremstyle{definition}

\theoremstyle{remark}
\newtheorem{remark}[theorem]{Remark}

\DeclareMathOperator{\hd}{hd}
\DeclareMathOperator{\cd}{cd}
\DeclareMathOperator{\gd}{gd}

\newcommand{\Q}{\mathbb Q}
\newcommand{\Z}{\mathbb Z}
\newcommand{\N}{\mathbb N}

\newcommand{\uuhd}{\mathop{\underline{\underline\hd}}}
\newcommand{\uucd}{\mathop{\underline{\underline\cd}}}
\newcommand{\uugd}{\mathop{\underline{\underline\gd}}}

\newcommand{\EXT}{\operatorname{Ext}}

\newcommand{\cohom}[3]{H^{{\raise1pt\hbox{$\scriptstyle#1$}}}(#2\>\!,#3)}
\newcommand{\tatecohom}[3]%
  {\widehat H^{{\raise1pt\hbox{$\scriptstyle#1$}}}(#2\>\!,#3)}

\newcommand{\Cohom}[3]%
  {H^{{\raise1pt\hbox{$\scriptstyle#1$}}}\big(#2\>\!,#3\big)}
\newcommand{\Tatecohom}[3]%
  {\widehat H^{{\raise1pt\hbox{$\scriptstyle#1$}}}\big(#2\>\!,#3\big)}

\newcommand{\homol}[3]{H_{{\lower1pt\hbox{$\scriptstyle#1$}}}(#2\>\!,#3)}
\newcommand{\homolog}[2]{H_{{\lower1pt\hbox{$\scriptstyle#1$}}}(#2)}

\newcommand{\IND}{\operatorname{Ind}}

\newcommand{\RES}{\operatorname{Res}}

\newcommand{\mono}{\rightarrowtail}
\newcommand{\epi}{\twoheadrightarrow}

\newcommand{\eg}{{\underline EG}}

\newcommand{\uueg}{{\underline{\underline E}}G}

\newcommand{\Ab}{\mathfrak{Ab}}

\newcommand{\OFG}{\mathcal O_{\mathcal F}G}

\newcommand{\frakF}{\mathfrak{F}}
\newcommand{\frakG}{\mathfrak{G}}

\DeclareMathOperator{\TOR}{Tor}

\catcode`\@=11

\newcommand{\calO}{\mathcal O}
\newcommand{\ModOFG}{\mathop{{\operator@font
Mod\text{-}}\calO_{\frakF}G}}
\newcommand{\OFGMod}{\mathop{\calO_{\frakF}G\text{-}{\operator@font
Mod}}}
\newcommand{\ModOGG}{\mathop{{\operator@font
Mod\text{-}}\calO_{\frakG}G}}
\newcommand{\OGGMod}{\mathop{\calO_{\frakG}G\text{-}{\operator@font
Mod}}}

\newcommand{\OfG}{\OC{G_{1}}{\frakF_{1}}}
\newcommand{\OgG}{\OC{G_{2}}{\frakF_{2}}}

\DeclareMathOperator{\GSet}{\mathit{G}-\mathfrak{Set}}

\newcommand{\Fall}{\frakF_{\operator@font all}}
\newcommand{\Ffin}{\frakF_{\operator@font fin}}
\newcommand{\Fvc}{\frakF_{\operator@font vc}}
\newcommand{\Fic}{\frakF_{\operator@font ic}}
\newcommand{\Ffg}{\frakF_{\operator@font fg}}
\newcommand{\Fpc}{\frakF_{\operator@font pc}}
\newcommand{\Fab}{\frakF_{\operator@font ab}}
\newcommand{\Fvpc}{\frakF_{\operator@font vpc}}
\newcommand{\Fvab}{\frakF_{\operator@font vab}}

\catcode`\@=12

\newcommand{\OC}[2]{\mathop{\mathcal{O}_{#2}#1}\nolimits}
\renewcommand{\OFG}{\OC{G}{\frakF}}
\newcommand{\OGG}{\OC{G}{\frakG}}

\DeclareMathOperator{\mor}{mor}

\newcommand{\longto}{\longrightarrow}

\DeclareMathOperator{\vcd}{vcd}
\DeclareMathOperator{\PD}{pd}
\DeclareMathOperator{\FLD}{fld}

\DeclareMathOperator{\id}{id}

\renewcommand{\coprod}%
{\mathop{\rotatebox[origin=c]{180}{$\displaystyle\prod$}}\limits}

\let\uZ\uz
\let\isom\iso

\newcommand{\uu}[1]{\underline{\underline{#1}}}

\DeclareMathOperator{\lcm}{lcm}

\renewcommand{\:}{\mathord{:}\hspace*{0.8ex plus .25ex minus .1ex}}

\title[Elementary amenable groups]{On the classifying space for the
family of virtually cyclic subgroups for elementary amenable groups}
\author{Martin G.~ Fluch} \address{Department of Mathematics,
Bielefled University, Postbox 100131, 33501 Bielefeld, Germany }
\email{mfluch@math.uni-bielfeld.de }

\author{Brita  E.~A.~Nucinkis}
\address{School of Mathematics, University of Southampton,
Southampton,
SO17 1BJ, United Kingdom}
\email{B.E.A.Nucinkis@soton.ac.uk}

\date{\today} 

\begin{document}

\begin{abstract} 
    We show that every elementary amenable group that has a bound on
    the orders of its finite subgroups admits a finite dimensional
    model for $\uueg$, the classifying space for actions with
    virtually cyclic isotropy.
\end{abstract}

\maketitle

\thispagestyle{empty}


\section{Introduction}

Classifying spaces with isotropy in a family have been studied for a
while; most of the research has focussed on $\eg$, the classifying
space with finite isotropy~\cites{lueckbook, lueck, luecksurvey}.
Finiteness conditions for $\eg$ for elementary amenable groups are
very well understood~\cites{fn, kmn}.  Finding manageable models for
$\uueg$ has been shown to be much more elusive.  In~\cite{jpl} it was
conjectured that the only groups admitting a finite type model for
$\uueg$ are virtually cyclic, and this was proved for hyperbolic
groups.  In~\cite{KoMN} it was shown that this conjecture also holds
for elementary amenable groups.  As far as finite dimensional models
are concerned, only a little more is known.  So far manageable models
have been found for crystallographic groups~\cite{lafontortiz},
polycyclic-by-finite groups~\cite{lueckweiermann}, hyperbolic groups
\cite{jpl} and ${\rm CAT(0)}$-groups~\cite{farley-10, luck-09}.
Adapting the construction of~\cite{jpl}, the first
author~\cite{fluch-11} has recently found a good model in the case
when $G$ is a certain type of HNN-extension, including extensions of
the form $G = A \rtimes \Z$, where the generator of $\Z$ acts freely
on the non-trivial elements of an abelian group $A$.  Utilising this
construction we prove:

\begin{Main Theorem}
    Let $G$ be an elementary amenable group with finite Hirsch length.
    If $G$ has a bound on the orders of its finite subgroups, then $G$
    admits a finite dimensional model for $\uueg$.
\end{Main Theorem}

The proof of this fact is algebraic, and uses finiteness
conditions in Bredon cohomology.  Bredon cohomology takes the place of
ordinary cohomology when studying classifying spaces with isotropy in
a family of subgroups.  We give a brief introduction into Bredon
cohomology in Section 2 and then move on to discussing dimensions in
Bredon cohomology for extensions, directed unions and direct products
of groups.  We also consider the behaviour of the Bredon cohomological
dimension when changing the family of subgroups.

A further crucial ingredient is Hillman--Linnell's and Wehrfritz'
\cites{hillmanlinnell, wehrfritz} characterisation of elementary
amenable groups as locally finite-by-soluble-by-finite groups.  This
allows us to reduce the problem to torsion-free abelian-by-cyclic
groups.  We show, Proposition \ref{ab-by-cyc}, that a torsion-free
abelian-by-cyclic group of finite Hirsch length has finite
Bredon-cohomological dimension bounded by a recursively defined
integer only depending on the Hirsch length.  In Section 6 this result
is extended to torsion-free nilpotent-by-abelian groups, which allows
us to prove the Main Theorem in Section 7.


\section{Background on Bredon cohomology}

In this article a \emph{family} $\frakF$ of subgroups of a group $G$
stands for a non-empty set of subgroups of $G,$ which is closed under
conjugation and taking finite intersections.  Common examples are the
\emph{trivial} family of subgroups $\frakF = \{1\}$, the family
$\Ffin(G)$ of all finite subgroups of $G$ and the family $\Fvc(G)$ of
all virtually cyclic subgroups of $G$.

Let $\frakF$ be a family of subgroups of $G$ and $K\leq G$ and put
\begin{equation*}
    \frakF\cap K = \{ H\cap K \mid H\in \frakF\},
\end{equation*}
which is a family of subgroups of $K$.  Now let $\frakF_{1}$ and
$\frakF_{2}$ be families of subgroups of some groups $G_{1}$ and
$G_{2}$ respectively.  Here we put
\begin{equation*}
    \frakF_{1}\times \frakF_{2} = \{ H_{1}\times H_{2} \mid 
    H_{1}\in \frakF_{1} \text{ and } H_{2}\in\frakF_{2}\}.
\end{equation*}
This is a family of subgroups of $G_{1}\times G_{2}$.  Finally, for
any family $\frakF$ of subgroups of $G$ we can define its
\emph{subgroup completion} $\bar\frakF$ as
\begin{equation*}
    \bar\frakF = \{ H\leq K \mid K\in \frakF\}.
\end{equation*}
That is, $\bar\frakF$ is the smallest family of subgroups of $G$ which
contains $\frakF$ and is closed under forming subgroups.

Given a non-empty $G$-set $X,$ we denote by $\frakF(X)$ the collection
of all its isotropy groups.  In general this is not a family of
subgroups as it may not be closed under finite intersections.

\medskip

Bredon cohomology was introduced for finite groups by
Bredon~\cite{bredon-67} and it has been generalised to arbitrary
discrete groups by Lück~\cite{lueckbook}.  It is the natural choice for a
cohomology theory to study classifying spaces with stabilisers in a
prescribed family~$\frakF$ of subgroups.  The reader is referred to
Lück's book~\cite{lueckbook} and the introductory chapters in Mislin's
survey \cite{mislin-03} for standard facts and definitions.  We shall,
however include those definitions and results on Bredon cohomology needed
later on.

\medskip

Given a group $G$, the \emph{orbit category $\calO G$} is defined as
follows: objects are the transitive $G$-sets~$G/H$ with $H\leq G$; the
morphisms of $\calO G$ are all $G$-maps $G/H \to G/K$.  For a family
$\frakF$ of subgroups of~$G$, the \emph{orbit category $\OFG$} is the
full subcategory of $\calO G,$ which has as objects the transitive
$G$-sets $G/H$ with $H\in \frakF$.

An \emph{$\OFG$-module}, or \emph{Bredon module}, is a functor $M\:
\OFG\to \Ab$ from the orbit category to the category of abelian
groups.  If the functor $M$ is contravariant, $M$ is said to be a
\emph{right} Bredon module; if $M$ is covariant we call it a
\emph{left} Bredon module.  A natural transformation $f\: M\to N$
between two $\OFG$-modules of the same variance is called a morphism
of $\OFG$-modules.  If $M$ is a right, respectively left, Bredon
module and $\varphi$ a morphism, then we may abbreviate $M(\varphi)$
by $\varphi^{*}$ and $\varphi_{*}$ respectively.

The \emph{trivial $\OFG$-module} is denoted by
$\underline\Z_{\frakF}.$ It is given by $\underline\Z(G/H) = \Z$ and
$\underline\Z(\varphi)=\id$ for all objects and morphisms of $\OFG$.

The categories of right, respectively left, $\OFG$-modules and their
morphisms are denoted by $\ModOFG$ and $\OFGMod$ respectively.  These
are functor categories and therefor inherit a number of properties
from the category $\Ab$.  For example, a sequence $L \to M\to N$ of
Bredon modules is exact if and only if when evaluated at every $G/H\in
\OFG,$ we obtain an exact sequence $L(G/H) \to M(G/H)\to N(G/H)$ of
abelian groups.

Since $\Ab$ has enough projectives, so does $\ModOFG$. Therefore we can
define homology functors in $\ModOFG$.  Denote by $\mor_{\frakF}(M,N)$ the morphisms between two Bredon modules $M$ and $N.$ Hence the
bi-functor
\begin{equation*}
    \mor_{\frakF}(?, ??)\: \ModOFG\times\ModOFG\to \Ab
\end{equation*}
has derived functors, denoted by $\EXT_{\frakF}^{*}(?,??)$.  The
categorical tensor product~\cite{schubert-70a}*{pp.~45ff.} gives rise
to a tensor product
\begin{equation*}
    ? \otimes_{\frakF} ??\: \ModOFG\times \OFGMod\to \Ab
\end{equation*}
\emph{over the orbit category $\OFG$}~\cite{lueckbook}*{p.~166}.  Its
derived functors are denoted by $\TOR^{\frakF}_{*}(?,??)$.

One can also define a tensor product over $\Z$ as follows: For
$\OFG$-modules $M$ and $N$ of the same variance, we set $(M\otimes
N)(G/K)=M(G/K)\otimes N(G/K),$ see also \cite{lueckbook}*{p.~166}.

\medskip

We shall now describe the basic properties of free, projective and
flat $\OFG$-modules.  Consider the following right Bredon module:
$\Z[?,G/K]_{G}$ with $K\in\frakF$.  Evaluated at $G/H$ this functor is
the free abelian group $\Z[G/H,G/K]_{G}$ on the set $[G/H,G/K]_{G}$ of
$G$-maps $G/H\to G/K$.  These modules are free,
cf.~\cite{lueckbook}*{p.~167}, and can be viewed as the building
blocks of the free right Bredon modules.  In general a free object in
$\ModOFG$ is of the form $\Z[?,X]_{G}$ where~$X$ is a $G$-set
with~$\frakF(X)\subset \frakF$.  Free left Bredon modules are defined
analogously: they are obtained from modules of the form
$\Z[G/K,??]_G$, where $K\in \frakF.$ A projective $\OFG$-modules is
then defined as a direct summand of a free module.

Note that the construction of free $\OFG$-modules is functorial in the
second variable. In particular, we have a functor
\begin{equation*}
    \Z[?,??]_{G}\: \GSet\to \ModOFG
\end{equation*}
defined by $\Z[?,??]_G(X)=\Z[?,X]_{G}$ for arbitrary $G$-sets $X$.

A Bredon module $M$ is \emph{finitely generated,} if it is the 
homomorphic image of a free Bredon module $\Z[?,X]_{G},$ where $X$ 
has only finitely many $G$-orbits.

A right $\OFG$-module $M$ is called \emph{flat} if it is
$?\otimes_{\frakF}N$-acyclic for every left $\OFG$-module $N$.  This
is the case if and only if the functor $\TOR^{\frakF}_{1}(M,?)$ is
trivial.  Flat Bredon modules share many properties with ordinary flat
modules.  In particular,

\begin{proposition}
    \cite{nucinkis-04}*{Theorem~3.2}
    \label{prop:Lazard-for-Bredon}
    A right $\OFG$-module $M$ is flat if and only if it is the 
    filtered colimit of finitely generated free $\OFG$-modules.
\end{proposition}

\medskip

Given a covariant functor $F\: \OfG \to \OgG$ between orbit
categories, one can now define induction and restriction functors
along $F$, see~\cite{lueckbook}*{p.~166}:
\begin{equation*}
    \arraycolsep0.2ex
    \begin{array}{lccc}
    \IND_{F}\: & \OfG & \to &\OgG \\
    & M(?) & \mapsto & M(?)\otimes_{\frakF_{1}} \mor_{\frakF_{2}}(??,F(?))
    \end{array}
\end{equation*}
and
\begin{equation*}
    \arraycolsep0.25ex
    \begin{array}{lccc}
    \RES_{F}\: & \OgG & \to &\OfG \\
    & M(??) & \mapsto & M\circ F(??)
    \end{array}
\end{equation*}

Since these functors are adjoint to each others, $\IND_{F}$ commutes
with arbitrary colimits~\cite{mac-lane-98}*{pp.~118f.} and preserves
free and projective Bredon modules~\cite{lueckbook}*{p.~169}.

\pagebreak[3]

\begin{lemma}
    \label{lem:induction-preserves-flats}
    Induction along $F$ preserves flat right Bredon modules.
\end{lemma}

\begin{proof}
    This follows  from the fact that both $\IND_{F}$ and
    $\TOR_{1}^{\frakF}(?,N)$ commute with filtered colimits and from
    Proposition~\ref{prop:Lazard-for-Bredon}.
\end{proof}

Let $\frakF \subset \frakG$ are two families of subgroups of a group 
$G$,  then the inclusion of the respective orbit categories is denoted by:
\begin{equation*}
    I\: \OFG\to \OGG
\end{equation*}
If $K$ is a subgroup of $G$ such that $\frakF\cap K$ is contained
in $\frakF$, then we consider the following functor
\begin{equation*}
    \arraycolsep0.2ex
    \begin{array}{lccc}
        I_{K}\: & \OC{K}{\frakF\cap K} &\to& \OFG \\
        & K/H & \mapsto & G/H.
    \end{array}
\end{equation*}

Note that for every non-empty $K$-sets $X$ with $\frakF(X)\subset
\frakF\cap K$ the functor $I_{K}$ can be extended by mapping each
$K$-orbit separately.

\begin{lemma}
    \cite{symonds-05}*{Lemma 2.9}
    \label{lem:symonds}
    Let $K$ be a subgroup of $H$ such that $\frakF\cap K$ is a 
    non-empty subset of $\frakF$. Then induction with $I_{K}$ is an 
    exact functor.
\end{lemma}

\medskip

We conclude this section with a collection of facts concerning
dimensions both generally and for the family of virtually cyclic
subgroups.  The \emph{Bredon cohomological dimension} $\cd_{\frakF}G$
of a group $G$ with respect to the family $\frakF$ of subgroups is the
projective dimension $\PD_{\frakF}\uZ$ of the trivial $\OFG$-module
$\uZ$.  Similarly, the \emph{Bredon homological dimension}
$\hd_{\frakF}G$ is the flat dimension $\FLD_{\frakF}\uZ$ of the
trivial $\OFG$-module.  For $\frakF = \Fvc(G)$ we shall use the
following notation: $\uucd G = \cd_{\frakF}G$ and $\uuhd G
=\hd_{\frakF}G.$

The cellular chain complex of a model for $E_{\frakF}G$ yields a free
resolution of the trivial $\OFG$-module
$\uZ$~\cite{lueckbook}*{pp.~151f.}.  In particular, this implies that
the Bredon geometric dimension $\gd_{\frakF}G$, the minimal dimension
of a model for $E_{\frakF}G$, is an upper bound for $\cd_{\frakF}G$.
Since projectives are flat this implies that 
\begin{equation*}
    \hd_{\frakF}G\leq \cd_{\frakF}G \leq \gd_{\frakF}G.
\end{equation*}

Furthermore, L\"uck and Meintrupp gave an upper bound of
$\gd_{\frakF}G$ in terms of $\cd_{\frakF}G$ as follows; the case
$\frakF =\Ffin(G)$ was shown in \cite{lueckbook}:

\begin{proposition}
    \cite{lm}*{Theorem~0.1~(i)}
    \label{prop:lm}
    Let $G$ be a group. Then
    \begin{equation*}
        \gd_{\frakF}G\leq\max(3,\cd_{\frakF}G).
    \end{equation*}
\end{proposition}

Hence, as long as  $\cd_{\frakF}G\geq 3$ or
$\gd_{\frakF}G\geq 4,$ we have equality of these two dimensions.

\medskip

Now suppose  $H$ is a subgroup of $G$ such that $\frakF\cap H$ is a 
non-empty subset of $\frakF$. Then 
\begin{equation*}
    \gd_{\frakF\cap H} H\leq \gd_{\frakF}G \quad \text{and} \quad
    \cd_{\frakF\cap H} H\leq \cd_{\frakF}G.
\end{equation*}

The following result is a consequence of Mart\'inez-P\'erez'
Lyndon--Hochschild--Serre spectral sequence in Bredon (co)homology
\cite{martinez}.  We shall only state the results for the family of
virtually cyclic subgroups.

\begin{proposition}
    \cite{martinez}*{Corollary 5.2}
    \label{prop:martinez}
    Let $N\mono G\epi Q$. Assume there exists $n\in \N$ such that 
    $\uucd H\leq n$ for every $N\leq H\leq G$ with $H/N$ virtually 
    cyclic. Then 
    \begin{equation*}
        \uucd G\leq n + \uucd Q.
    \end{equation*}
\end{proposition}

A careful inspection of the terms of the spectral
sequence~\cite{martinez}*{Theorem 4.3} yields the following:

\begin{proposition}
    \label{prop:finkernel}
    Let $F \mono G \epi Q$ be a group extension with $F$ finite.  Then
    \begin{equation*}
        \uucd G = \uucd Q.
    \end{equation*}
\end{proposition}

\begin{proof}
    The proof is identical to that for the corresponding result for
    the family of all finite subgroups \cite{nucinkis-04}*{Theorem
    5.5}.  One checks that the families in question satisfy the
    conditions of~\cite{martinez}*{Corollary 4.5}.
\end{proof}

Now suppose $G$ is a finite extension of a group $H$.  Lück has
constructed a model for $\uu EG$ from a model for $\uu EH$
\cite{lueck}.  This yields the following bound for $\uugd G$:

\begin{proposition}
    \label{prop:gd-for-finite-index}
    \cite{lueck}*{Theorem~2.4} Let $H$ be a finite index subgroup of
    $G$.  Then 
    \begin{equation*}
	\uugd G\leq |G:H| \cdot \uugd H.
    \end{equation*}
    In particular, $\uugd G$ is finite if and only if $\uugd H$ is
    finite.
\end{proposition}

In light of Proposition~\ref{prop:martinez} one needs to understand
the behaviour of the Bredon dimensions for the family of virtually
cyclic subgroups under extensions with virtually cyclic quotients.
In~\cite{fluch-11} the first author gave bounds for certain infinite
cyclic extensions:

\begin{proposition}
    \label{prop:fluch}
    \cite{fluch-11}*{Theorem~15}
    Let $G = B\rtimes \Z$ and assume that $\Z$ acts freely via 
    conjugation on the conjugacy classes of non-trivial elements 
    of~$B$. Then
    \begin{equation*}
        \uugd G \leq \uugd B + 1.
    \end{equation*}
\end{proposition}


\section{Directed unions of groups}

The standard resolution of $\Z$ in classical group cohomology
~\cite{brown-82}*{pp.~15f} has been extended to Bredon cohomology for
the family $\Ffin(G)$ of all finite subgroups of a given group $G$
\cite{nucinkis-04}.  This construction can be generalised to arbitrary
families~$\frakF$ without any essential changes:

For each $n\in \N$ let $\Delta_{n}$ be the $G$-set
\begin{equation*}
    \Delta_{n}= 
    \{ (g_{0}K_{0}, \ldots, g_{n}K_{n}) \mid g_{i}\in G \text{ and }
    K_{i}\in \frakF\}.
\end{equation*}
Since $\frakF$ is closed under taking finite intersections
it follows that $\frakF(\Delta_{n}) \subset \frakF$.  For $n\geq 1$
and $0\leq i\leq n$ we define $G$-maps $\partial_{i}\: \Delta_{n}\to
\Delta_{n-1}$ by
\begin{equation*}
    \partial_{i}(g_{0}K_{0},\ldots, g_{n}K_{n}) = (g_{0}K_{0}, 
    \ldots, \widehat{g_{i}K_{i}}, \ldots, g_{n}K_{n})
\end{equation*}
where $(g_{0}K_{0}, \ldots, \widehat{g_{i}K_{i}}, \ldots,
g_{n}K_{n})$ 
denotes the $n$-tuple obtained from the $(n+1)$-tuple 
$(g_{0}K_{0},\ldots, g_{n}K_{n})$ by deleting the $i$-th component.

Let $\Delta_{-1} = \{*\}$ be the singleton set with 
trivial $G$-action.  The unique map $\varepsilon\: \Delta_{0}\to
\Delta_{-1}$ is obviously $G$-equivariant.  Also note that
$\Z[?,\Delta_{-1}]_{G} = \uZ$.

We now obtain a resolution of the trivial $\OFG$-module $\uZ_{\frakF}$
by right $\OFG$-modules
\begin{equation*}
    \ldots \longto \Z[?, \Delta_{2}]_{G} \stackrel{d_{2}}{\longto}
\Z[?,
    \Delta_{1}]_{G} \stackrel{d_{1}}{\longto} \Z[?, \Delta_{0}]_{G}
    \stackrel{\varepsilon^{*}}{\epi} \uZ
\end{equation*}
where
\begin{equation*}
    d_{n} = \sum_{i=0}^{n} (-1)^{i} \partial_{i}^{*}.
\end{equation*}
Since $\frakF(\Delta_{n}) \subset \frakF$ it follows that this
resolution is free and it is called the \emph{standard resolution} of 
the trivial $\OFG$-module $\uZ$.

There now follows a variation of  \cite{nucinkis-04}*{Theorem~4.2}.

%
%

\begin{proposition}
    \label{prop:directunions}
    Let $G$ be a directed union of subgroups $G_\lambda$, where
    $\lambda \in \Lambda$ is some indexing set.  Let $\frakF$ be
    family of subgroups of $G$.  For each $\lambda \in \Lambda$ put
    $\frakF_{\lambda}= \frakF\cap G_{\lambda}$ and suppose that $\frakF =
    \bigcup_{\lambda\in \Lambda} \frakF_{\lambda}$. Then
    \begin{enumerate}
        \item  $\hd_{\frakF} G = \sup\{ \hd_{\frakF_{\lambda}} 
	G_{\lambda} \}$.
    
	\item If $\Lambda$ is countable then $\cd_{\frakF} G \leq
	\sup\{ \cd_{\frakF_{\lambda}} G_{\lambda} \} + 1$.
    \end{enumerate}
\end{proposition}

\noindent Note that Part (ii) has also been derived in \cite[Corollary 4.3]{DPT}, using a spectral sequence argument.

\begin{corollary}
    \label{cor:directunions}
    Let $G$ and $G_{\lambda}$, $\lambda\in\Lambda$, as in 
    Proposition~\ref{prop:directunions}. Then
    \begin{enumerate}
        \item  $\uuhd G = \sup \{ \uuhd G_{\lambda}\}$.
    
	\item If $\Lambda$ is countable then $\uucd G \leq \sup \{
	\uucd G_{\lambda}\} + 1$.
    \end{enumerate}
\end{corollary}

\begin{proof}
  This follows from the fact that $\Fvc(G_{\lambda}) = \Fvc(G)\cap
  G_{\lambda}$ and that for every finitely generated subgroup $H$
  there is a $\lambda\in \Lambda$ such that $H \in G_\lambda.$ Now
  apply Proposition~\ref{prop:directunions}.
\end{proof}

In particular, Corollary \ref{cor:directunions} (ii) can be applied to
countable groups.  A countable group is the direct union of its
finitely generated subgroups $G_{\lambda}$, $\lambda\in\Lambda$, where
$\Lambda$ is countable.  Hence $\uucd G \leq \sup \{ \uucd
G_{\lambda}\} + 1.$

\medskip

Before we can prove Proposition~\ref{prop:directunions}, we need the
following technical lemma.

\begin{lemma}
    \label{lem:directunions-aux}
    \begin{enumerate}
	\item Assume that the $G$-set $X$ is the direct union of 
	$G$-sets $X_{\alpha}$.  Then the homomorphism 
	\begin{equation}
	    \varinjlim \Z[?, X_{\alpha}]_{G} \to \Z[?,X]_{G}
	    \label{eq:direct-unions-1}
	\end{equation}
	induced by the canonical inclusions $\Z[?,X_{\alpha}]_{G}
	\hookrightarrow \Z[?, X]_{G}$ is an isomorphism.
    
        \item The homomorphism 
	\begin{equation}
	    \varinjlim \Z[?,G/G_{\lambda}]_{G} \to \Z[?,G/G]_{G}
	    \label{eq:direct-unions-2}
	\end{equation}
	induced by the projections $G/G_{\lambda}\epi G/G$ is an 
	isomorphism.
	
	\item Let $K\leq G$ such that $\frakF\cap K \subset \frakF$.
	If $X$ is a $K$-set with $\frakF(X)\subset\frakF\cap K$, then
	\begin{equation*}
	    \IND_{I_{K}} \Z[?,X]_{K} \isom \Z[?, I_{K}(X)]_{G},
	\end{equation*}
	and this isomorphism is natural in $X$.
    \end{enumerate}
\end{lemma}

\begin{proof}
    (i) Let $H\in \frakF$ and evaluate~\eqref{eq:direct-unions-1} at
    $G/H$.  The inclusion $X_{\alpha}^{H}\hookrightarrow X^{H}$
    induces a homomorphism
    \begin{equation}
        \varinjlim \Z[X_{\alpha}^{H}] \to \Z[X^{H}].
	\label{eq:direct-unions-3}
    \end{equation}
    $X_{\alpha}^{H} = X_{\alpha}\cap X^{H}$, implying that $\varinjlim
    X_{\alpha}^{H} = X^{H}$.  Since $\Z[?]$ commutes with colimits it
    follows that~\eqref{eq:direct-unions-3} is an isomorphism.  Hence
    \eqref{eq:direct-unions-1} is an isomorphism of $\OFG$-modules.
    
   \medskip
   
   (ii) This follows directly from the universal 
   property of a colimit.
   
   \medskip
   
   (iii) Let $R$ be a complete system of representatives of the orbit
   space $X/K$.  Then we have the following sequence of isomorphisms
   of $\OFG$-modules:
   \begin{align*}
       \IND_{I_{K}} \Z[?,X]_{K}
       & \isom
       \coprod_{x\in R} \IND_{I_{K}} \Z[?, K/K_{x}]_{K}
       \\
       & \isom
       \coprod_{x\in R} \bigl( \Z[??,K/K_{x}]_{K} \otimes_{\frakF\cap
       K} \Z[?,I_{K}(??)]_{G}\bigr) 
       \\
       & \isom \coprod_{x\in R}\Z[?, I_{K}(K/K_{x})]_{G}  \\
     &  \isom \coprod_{x\in R}\Z[?, G/K_{x}]_{G} \\
    &   \isom \Z[?,I_{K}(X)]_{G}
   \end{align*}
   Note that the third isomorphism is a consequence of the
   Yoneda-Lemma and that the composition of these isomorphisms is
   clearly natural in~$X$.
\end{proof}

\begin{proof}[Proof of Proposition~\ref{prop:directunions}]
    For each $\lambda\in \Lambda$ we have the standard resolution of
    $\OC{G_{\lambda}}{\frakF_{\lambda}}$-modules
    \begin{equation}
	\label{eq:direct-unions-4a}	
        \ldots
	\to \Z[?,\Delta_{\lambda,2}]_{G_{\lambda}}
	\to \Z[?,\Delta_{\lambda,1}]_{G_{\lambda}}
	\to \Z[?,\Delta_{\lambda,0}]_{G_{\lambda}}
	\epi\Z[?]_{G_{\lambda}}.
    \end{equation}
    By Lemma~\ref{lem:symonds} the functor $\IND_{I_{G_{\lambda}}}$ is
    exact.  Hence for each $\lambda\in \Lambda$ there is an exact
    sequence of $\OFG$-modules:
  
    \begin{equation}
        \ldots
	\to \Z[?,X_{\lambda,2}]_{G}
	\to \Z[?,X_{\lambda,1}]_{G}
	\to \Z[?,X_{\lambda,0}]_{G}
	\epi\Z[?,G/G_{\lambda}]_{G},
	\label{eq:direct-unions-4}
    \end{equation}
    where $X_{\lambda,n} = I_{G_{\lambda}}(\Delta_{\lambda,n})$.  Note
    that the $X_{\lambda,n}$ are $G$-invariant subsets of $\Delta_{n}$
    and that $\Delta_{n}$ is the directed union of the
    $X_{\lambda,n}$.  For each $\lambda\leq \mu$ the inclusion
    $X_{\lambda,n}\hookrightarrow X_{\mu,n}$, $n\geq 0$, induces a
    homomorphism
    \begin{equation*}
	\eta^{\mu}_{\lambda,n}\: \Z[?, X_{\lambda,n}]_{G} \to \Z[?,
	X_{\mu,n}]_{G}.
    \end{equation*}
    Also, the projection $G/G_{\lambda}\epi G/G_{\mu}$ induces 
    homomorphisms
    \begin{equation*}
        \eta^{\mu}_{\lambda,-1}\: \Z[?,G/G_{\lambda}]_{G} \to 
	\Z[?,G/G_{\mu}]_{G}.
    \end{equation*}
    Hence we have chain-maps
     between the corresponding chain
    complexes~\eqref{eq:direct-unions-4}.  These chain
    complexes together with the chain maps
    $\eta_{\lambda,*}^{\mu}$ form a direct limit system indexed
    by $\Lambda$. 
    Lemma~\ref{lem:directunions-aux}~(i) and~(ii) imply that its limit is
    the sequence
    \begin{equation}
        \ldots\to \Z[?, \Delta_{2}]_{G} \to \Z[?, \Delta_{1}]_{G} \to 
	\Z[?, \Delta_{0}]_{G} \epi \uZ_{\frakF}.
	\label{eq:direct-unions-4c}
    \end{equation}
    Since direct limits preserve 
    exactness~\cite{weibel-94}*{p.~57}, this sequence is exact.
    
    \medskip
    
    Denote by $K_{\lambda,n}$ the $n$-th kernel of the
    sequence~\eqref{eq:direct-unions-4a}.  As before,
    $\IND_{G_{\lambda}}(K_{\lambda,n})$ is the $n$-th kernel
    in~\eqref{eq:direct-unions-4} and the chain maps
    $\eta_{\lambda,n}^{\mu}$ yield a inverse limit system.  Since
    taking direct limits preserves exactness we get that
    \begin{equation*}
        K_{n}= \varinjlim(\IND_{G_{\lambda}} K_{\lambda,n})
    \end{equation*}
    is the $n$-th kernel of~\eqref{eq:direct-unions-4c}.
    
    \medskip
    
    Now suppose that there exists a $n\in \N$ such that
    $\hd_{\frakF_{\lambda}} G_{\lambda}\leq n$ for all $\lambda\in
    \Lambda$.  In particular, all $K_{\lambda,n}$ are flat.  Now
    Lemma~\ref{lem:induction-preserves-flats} implies that
    $\IND_{G_{\lambda}} K_{\lambda,n}$ all are flat.  Since
    $\TOR^{\frakF}_{1}(?,M)$ commutes with direct limits it follows
    that $K_{n}$ is a flat $\OFG$-module.  In particular,
    $\hd_{\frakF}G \leq n$ proving (i).
    
    The proof of (ii) is analogous.
    Apply~\cite{nucinkis-04}*{Lemma~3.4}, which states that a
    countable colimit of projective $\OFG$-modules has projective
    dimension~$\leq 1$.
\end{proof}

\begin{lemma}
    \label{lem:cd-A}
    Let $A$ be a countable abelian group with finite Hirsch length
    $h(A)$.  Then
    \begin{equation*}
        \uucd A \leq h(A)+2.
    \end{equation*}
\end{lemma}

\begin{proof} 
    Write $A$ as the countable direct union $A=\varinjlim A_{\lambda}$
    of its finitely generated subgroups $A_{\lambda}$.
    \cite{lueckweiermann}*{Theorem~5.13} implies $\uugd
    A_{\lambda}\leq h(A_{\lambda})+1$, and hence $\uucd
    A_{\lambda}\leq h(A_{\lambda}) + 1\leq h(A) + 1$.  Thus, by
    Corollary~\ref{cor:directunions}, $\uucd A \leq h(A)+2$ as
    required.
\end{proof}


\medskip

%
%

\section{Change of family and direct products of groups}

The following result is the algebraic counterpart
to~\cite{lueckweiermann}*{Proposition~5.1~(i)}.  Although we only
state and prove it for Bredon cohomology, an analogous statement also
holds for Bredon homology. The result for Bredon cohomology has also been proved in \cite[Corollary 4.1]{DPT} using a spectral sequence argument.

\begin{proposition}
    \label{prop:cd-for-larger-families}
    Let $G$ be a group and $\frakF$ and $\frakG$ two families of
    subgroups of $G$ such that $\frakF\subset \frakG$ and that, for
    every $K\in \frakG$, $\frakF\cap K \subset\frakF$.  Suppose there
    exists $k\geq 0$ such that, for every $K\in \frakG$,
    $\cd_{\frakF\cap K} K \leq k$.  Then
        \begin{equation*}
        \cd_{\frakF} G \leq \cd_{\frakG} G + k.
    \end{equation*}
\end{proposition}

\begin{proof}
    Let $I\: \OFG\hookrightarrow \OGG$ be the inclusion functor.  We
    begin by showing that for every projective $\OGG$-module $P$,
    $\PD_{\frakF}(\RES_{I} P) \leq k$. Since restriction is an
    exact additive functor, it it suffices to prove this claim for $P
    = \Z[?, G/K]_{G}$ where $K\in \frakG$.
    
    Since $\cd_{\frakF\cap K} K\leq k$ there exists a projective 
    resolution  
    \begin{equation*}
	0 \to P_{k} \to \ldots \to P_{0}\to \uZ_{\frakF\cap
	K}\to 0
    \end{equation*}
    of the trivial $\OC{K}{\frakF\cap K}$-module $\uZ_{\frakF\cap K}$.
    Since induction with $I_{K}$ is exact, see
    Lemma~\ref{lem:symonds}, and preserves projectives, we obtain a
    projective resolution
    \begin{equation*}
	0\to \IND_{I_{K}} P_{k} \to \ldots \to \IND_{I_{K}} P_{0}\to
	\Z[?, G/K]_{G}\to 0
    \end{equation*}
    of length $k$ of the $\OFG$-module $\Z[?, G/K]_{G}$.  However, by
    \cite{symonds-05}*{Lemma~2.7}, $\Z[?, G/K]_{G} \isom \RES_{I}P$
    implying $\PD_{\frakF}(\RES_{I} P) \leq k$ as claimed.
    
    \medskip
    
    Now $\cd_{\frakG} G=n$.  Then there exists a projective
    resolution
    \begin{equation*}
        0 \to P_{n} \to \ldots\to P_{0} \to \uZ_{\frakG} \to 0
    \end{equation*}
    of the trivial $\OGG$-module $\uZ_{\frakG}$.  Upon restriction we
    obtain a resolution
    \begin{equation}
	0 \to \RES_{I} P_{n} \to \ldots \to \RES_{I} P_{0} \to
	\uZ_{\frakF} \to 0
	\label{eq:cd-for-larger-families}
    \end{equation}
    of the trivial $\OFG$-module $\uZ_{\frakF}$ by $\OFG$-modules of
    projective dimension at most $k.$ The result now follows by a
    dimension shifting argument.
\end{proof}

\begin{proposition}
    \label{prop:directproduct}
    Let $G_{1}$ and $G_{2}$ be groups and let $\frakF_{1}$ and
    $\frakF_{2}$ be subgroup-closed families of subgroups of $G_{1}$
    and $G_{2}$ respectively.  Let $G= G_{1}\times G_{2}$ and $\frakF=
    \frakF_{1}\times \frakF_{2}$ and take $\frakG\subset \bar\frakF$
    to be a subgroup-closed family of subgroups of $G.$ Assume that
    there exists $k\in \N$ such that $\cd_{\frakG\cap K} K \leq k$ for
    every $K\in \frakF$.  Then
    \begin{equation*}
        \cd_{\frakG} G \leq \cd_{\frakF_{1}} G_{1} + \cd_{\frakF_{2}} 
	G_{2} + k.
    \end{equation*}
\end{proposition}

Similar results for the families $\frakF_1=\frakF_2=\Ffin$ and
$G$-CW-complexes have been obtained in~\cite{sanchez,leonardi}.

\begin{corollary}
    \label{cor:directproduct}
    Let $G= G_{1}\times G_{2}$. Then $\uucd G\leq \uucd 
    G_{1} + \uucd G_{2} + 3$.
\end{corollary}

\begin{proof}
    Every $K\in \Fvc(G_{1})\times \Fvc(G_{2})$ is virtually 
    polycyclic of $\vcd K\leq 2$. Thus $\uucd K \leq \uugd K \leq 
    3$ by~\cite{lueckweiermann}*{Theorem~5.13} and the result follows 
    from Proposition~\ref{prop:directproduct}.
\end{proof}

\begin{remark}\label{rem:sharp}
    The bound in Corollary \ref{cor:directproduct} is sharp.  For
    example, $\uucd(\Z\times\Z)=3$, which follows from the proof of
    \cite[Theorem 5.12 (iii)]{lueckweiermann}, see also
    \cite[Corollary 4.3]{fluch-thesis}.
\end{remark}

\begin{remark}
    \label{lem:directproduct-aux}
    Suppose $G = G_{1}\times G_{2}$ and $\frakF= \frakF_{1}\times
    \frakF_{2}$ are as in Proposition~\ref{prop:directproduct}.  Let
    $X_{i}$ and $Y_{i}$ be $G_{i}$-sets and let $X= X_{1}\times X_{2}$
    and $Y = Y_{1}\times Y_{2}$ be $G$-sets with the obvious
    $G$-action.  We denote by $p_{i}\: G\to G_{i}$ be the canonical
    projections.  Since $\frakF_{1}$ and $\frakF_{2}$ are assumed to
    be closed under forming subgroups it follows that for every $H\in
    \bar\frakF$, $p_{1}(H)\times p_{2}(H)\in \frakF$.  Hence
    homomorphism
    \begin{equation*}
        f\: \Z[?,X]_{G} \to \Z[?,Y]_{G}
    \end{equation*}
    of right $\OFG$-modules extends to a homomorphism $f\: \Z[?,X]_{G}
    \to \Z[?,Y]_{G}$ of $\OC{G}{\bar\frakF}$-modules as follows: for
    every $H\in \bar\frakF$ let $f_{H} = f_{p_{1}(H)\times p_{2}(H)}.$
\end{remark}

Also note that the natural projections $p_i: G \to G_i$ give rise to
functors
\begin{equation*}
    p_i\: \OFG \to \OC{G_i}{\frakF_i}.
\end{equation*}

\begin{proof}[Proof of Proposition~\ref{prop:directproduct}]    
    Let $P_{*}\epi \uZ_{\frakF_{1}}$ and $Q_{*}\epi\uZ_{\frakF_{2}}$
    be free resolutions.  Hence there exist $G_{1}$-sets $X_{i}$ and
    $G_{2}$-sets $Y_{j}$ such that $P_{i} = \Z[?, X_{i}]_{G_{1}}$ and
    $Q_{j} = \Z[?, Y_{j}]_{G_{2}}$.
    
    Let $P'_{*} = \RES_{p_{1}}P_{*}$ and $Q'_{*} = \RES_{p_{2}}
    Q_{*}$.  These $\OFG$-modules are of the form
    \begin{equation*}
        P'_{i} = \Z[?, X_{i}]_{G}
	\qquad \text{and} \qquad
	Q'_{j} = \Z[?, Y_{i}]_{G},
    \end{equation*}
    where the action of $G$ on $X_{i}$ and $Y_{i}$ is given by $gx =
    p_{1}(g)x$ and $gy = p_{2}(g)y$ respectively.  For each $i,j\in
    \N$ we have an identification of $\OFG$-modules
    \begin{equation*}
        P'_{i}\otimes Q'_{j} = \Z[?, X_{i}\times Y_{j}]_{G}.
    \end{equation*}
    Here $G$ acts diagonally  on $X_{i}\times Y_{j}$. 
   
    This gives rise to a double complex in the usual way, see for
    example~\cite{weibel-94}*{pp.~58f.}.  Denote by $C_{k}$ its total
    complex:
    \begin{equation*}
        C_{k} = \coprod_{i=0}^{k}  \Z[?, X_{i}\times Y_{k-i}]_{G}.
    \end{equation*}
    The augmentation maps $\varepsilon_{1}\: P_{0}\epi
    \uZ_{\frakF_{1}}$ and $\varepsilon_{2}\: Q_{0}\epi
    \uZ_{\frakF_{1}}$ induce an augmentation map $\varepsilon\: 
    C_{0}\epi\uZ_{\frakF}$. Altogether we obtain a resolution
    \begin{equation}
	\ldots \to C_{2} \to C_{1}\to C_{0} \epi \uZ_{\frakF}
	\label{eq:proof-directproduct-1}
    \end{equation}
    of the trivial $\OFG$-module by free $\OFG$-modules.
    
    Now the free $\OFG$-modules $C_{k}$ are also free
    $\OC{G}{\bar\frakF}$-modules.  Since the families $\frakF_{1}$ and
    $\frakF_{2}$ are assumed to be subgroup closed we can extend,
    using Remark~\ref{lem:directproduct-aux}, every morphism in the
    sequence~\eqref{eq:proof-directproduct-1} to a morphism of the
    corresponding $\OC{G}{\bar\frakF}$-modules.  It follows that we
    obtain a resolution
    \begin{equation}
	\ldots \to C_{2} \to C_{1}\to C_{0} \epi \uZ_{\bar\frakF}
	\label{eq:proof-directproduct-2}
    \end{equation}
    of the trivial $\OC{G}{\bar\frakF}$-module by free 
    $\OC{G}{\bar\frakF}$-modules.
    
    Now assume that $m= \cd_{\frakF_{1}} G_{1}$ and $n=
    \cd_{\frakF_{2}}G_{2}$.  Then it follows from an Eilenberg Swindle
    that there are free resolutions $P_{*}\epi \uZ_{\frakF_{1}}$ and
    $Q_{*} \epi \uZ_{\frakF_{2}}$ as above, of lengths $m$ and $n$
    respectively.  This implies that $C_{k}=0$ for all $k> m+n.$ In
    particular
    \begin{equation*}
        \cd_{\bar \frakF} G \leq \cd_{\frakF_{1}} G_{1} + 
	\cd_{\frakF_{2}}G_{2}.
    \end{equation*}
    
    Let $K\in \bar\frakF$.  Then $K\leq K_{1}\times K_{2}$ for some
    $K_{1}\times K_{2}\in \frakF$.  Since $\frakG$ is assumed to be
    closed under forming subgroups it follows that $\emptyset\neq
    \frakG\cap K \subset \frakG\cap (K_{1}\times K_{2}).$ Therefore we
    have $\cd_{\frakG\cap K} K \leq \cd_{\frakG\cap(K_{1}\times
    K_{2})} (K_{1}\times K_{2})$.  By assumption the latter is bounded
    by $k$.  Thus we have $\cd_{\frakG}G \leq \cd_{\bar\frakF} G + k$
    by Proposition~\ref{prop:cd-for-larger-families} and the claim of
    the proposition follows.
\end{proof}

\begin{remark}
    Note that the special case of Corollary~\ref{cor:directproduct}
    follows almost immediately by applying Mart\'inez-P\'erez'
    spectral sequence Proposition~\ref{prop:martinez} twice, but we
    have included the above for its generality and for being rather
    elementary.  The alternative argument is as follows: Consider
    $G=G_{1}\times G_{2}$ as an extension
    \begin{equation*}
        G_{1} \mono G \epi G_{2}
    \end{equation*}
    By Proposition~\ref{prop:martinez} we have $\uucd G \leq m + \uucd
    G_{2}$ where $m$ is the supremum of $\uucd H$ where $H$ ranges
    over all $H\leq G$ such that $G_{1}\leq G$ such that $H/G_{1}$ is
    virtually cyclic.  But these $H\leq G$ are of the form $H =
    G_{1}\times V$ with $V$ a virtually cyclic subgroup of~$G_{2}$.
   This gives rise to an extension
    \begin{equation*}
        V\mono H \epi G_{1}.
    \end{equation*}
    Applying Proposition~\ref{prop:martinez} again yields that $\uucd 
    H\leq n + \uucd G_{1}$ where $n$ is the supremum of $\uucd L$ 
    where $L$ ranges over all $L\leq H$ with $V\leq L$ and $L/V$ is 
    virtually cyclic. These $L$ are of the form $V\times W$ with $W$ 
    a virtually cyclic subgroup of $G_{1}$. Thus
    \begin{equation*}
        \uucd G \leq k + \uucd G_{1} + \uucd G_{2}
    \end{equation*}
    where $k$ is the supremum of $\uucd (V_{1}\times V_{2})$ with
    $V_{1}$ and $V_{2}$ ranges over all virtually cyclic subgroups of
    $G_{1}$ and $G_{2}$ respectively.  In particular $k\leq 3$
    by~\cite{lueckweiermann}*{Theorem~5.13}.
\end{remark}

%
%

\section{Infinite cyclic extensions of abelian groups}

\begin{lemma}\label{modoutAbar}
    Let $G$ be a torsion-free abelian-by-(infinite cyclic) group, 
    i.e.~there is a short exact sequence
    \begin{equation*}
        A \mono G \epi \langle t \rangle
    \end{equation*}
    with $A$ abelian and $\langle t \rangle \cong \Z.$ Consider the
    subgroup $\bar A =\{a\in A \mid a^t=a\}$.  Then $G/\bar A$ is
    torsion-free.
\end{lemma}

\begin{proof}
    $\bar A$ is obviously a central subgroup in $G$ and we have a
    short exact sequence
    \begin{equation*}
	A/\bar A \mono G/\bar A \stackrel{\pi}{\epi} \langle t\rangle.
    \end{equation*}
    To prove the claim it suffices to show that $A/\bar A$ is
    torsion-free.

    Suppose there is an $a \in A$ such that $\pi(a)^n=\pi(a^n)=1.$
    This implies that $a^n \in \bar A$ and hence $(a^n)^t=a^n$.  Since
    $A$ is abelian, we have $(a^ta^{-1})^n =1.$ But $A$ is
    torsion-free and hence $a^t = a$.  This implies $a \in \bar A$.
\end{proof}

In a torsion-free abelian-by-(infinite cyclic) group, the generator
$t$ of the infinite cyclic group acts by automorphisms on the abelian
group $A$.  As we will see in Proposition~\ref{ab-by-cyc}, there is no
problem if $t$ acts trivially or freely on the non-trivial elements of
$A$.  The main problem arises when $t$ acts by a finite order
automorphism.  But the following, folklore, version of Selberg's Lemma
tells us that the order of this automorphism has a bound only
depending on $A$.  We shall state the Lemma as a special case
of~\cite{wehrfritz70}*{Theorem~T1}:

\begin{lemma}\label{selberg}\cite{wehrfritz70}
    Let $\Gamma$ be a group of automorphisms of a torsion-free abelian
    group $A$ of finite Hirsch length.  Then there exists an integer
    $m(A)$ such that every periodic subgroup of $\Gamma$ has order at
    most $m(A)$.
\end{lemma}

For our purpose we need the following consequence of Selberg's Lemma,
which is probably known.  We include it for completeness.

\begin{lemma}
    \label{lem:selberg-implication}
    Let $A$ be a torsion-free abelian group with finite Hirsch length
    $h(A)$.  Then there exists a integer $\nu= \nu(A)$ which depends
    only on $h(A)$ such that for any automorphism $t$ of $A$ the
    finite orbits of elements in $A$ under the action of $t$ have at
    most length $\nu$.
\end{lemma}

\begin{proof} 
    Let $n= h(A)$. Then $A\otimes\Q \isom \Q^{n}$ and $A$ can
    be viewed as an additive subgroup of $\Q^{n}$ by $a\mapsto
    a\otimes 1$.  The automorphism $t$ of $A$ extends to an
    automorphism $t\otimes \id\: \Q^{n}\to \Q^{n}$ of
    $\Q$-vector spaces, which we denote by $\varphi$.
    
    Let $U = \{a\in A \mid \varphi^{k}(a) = a \text{ for some $0\neq
    k\in\Z$}\}$.  Then $U$ is a $\varphi$-invariant subspace of
    $\Q^{n}$.  It has a complement $V$ in $\Q^{n}$ and there exists a
    unique linear map $\psi\: \Q^{n}\to \Q^{n}$ which agrees with
    $\varphi$ on $U$ and which is the identity on~$V$.  Then $\psi$ is
    an isomorphism which is periodic by construction. 
    
    By Selberg's Lemma there exists a number $\nu(n)$ such that every
    periodic automorphism of $\Q^{n}$ has order at most~$\nu(n)$.
    Therefore $\psi^{m} = \id$ for some $1\leq m\leq \nu(n)$.  In
    particular we have that $\varphi^{m}(a)= a$ for each $a\in U\cap
    A$.
\end{proof}

\begin{proposition}\label{ab-by-cyc}
    Let $A$ be a torsion-free abelian group of finite Hirsch
    length~$h$.  There is a recursively defined integer $f(h)$
    depending only on $h$ such that for every infinite cyclic
    extension $G=A\rtimes \langle t\rangle$ we have
    \begin{equation*}
	\uucd G\leq f(h).
    \end{equation*}
\end{proposition}

\begin{proof}
    We prove the proposition by induction on the Hirsch length of $A$.
    Since $A$ is torsion-free, $h=0$ implies that $A$ is trivial.  In
    this case $G$ is infinite cyclic and therefore $f(0)=0.$

    Now suppose $h\geq 1$ and assume that the statement is true for
    all torsion-free abelian groups $B$ with $h(B)<h$.  Let $A$ be a
    torsion-free abelian group with Hirsch length~$h$ and let $\bar A$
    be as in Lemma~\ref{modoutAbar}.  Then precisely one of the
    following three cases occurres.
    
    \medskip
    
    (1) \textsl{$\{1\}\neq \bar A$:} As in the proof of 
    Lemma~\ref{modoutAbar} we have short exact sequenc
    \begin{equation*}
        A/\bar A \mono G/\bar A \epi \langle t\rangle
    \end{equation*}
    with $A/\bar A$ torsion-free. Since $A$ is torsion-free and $\bar
    A\neq \{1\}$ it follows that $h(\bar A)\geq 1$ and thus $h(A/\bar
    A) < h$.  Then
    \begin{equation*}
        \uucd (G/\bar A) \leq f(h-1)
    \end{equation*}
    by induction.
    
    Consider the short exact sequence
    \begin{equation*}
        \bar A\mono G\mono G/\bar A.
    \end{equation*}
    We use Proposition~\ref{prop:martinez} to find a bound for $\uucd
    G$.  Let $S$ be a subgroup of $G$ such that $\bar A\leq S$ and
    $S/\bar A$ is virtually cyclic.  Since $G/\bar A$ is torsion free,
    $S/\bar A =\langle s\rangle$ is infinite cyclic.  Then the fact
    that $\bar A$ is a central subgroup of $G$ implies that
    \begin{equation*}
        S = \bar A \times \langle s\rangle.
    \end{equation*}
    In particular $S$ is countable abelian with finite Hirsch length
    $h (S) = h(\bar A) +1$.  Therefore we have $\uucd S\leq h(\bar
    A)+3$ by Lemma~\ref{lem:cd-A}.  Hence
    Proposition~\ref{prop:martinez} gives
    \begin{align*}
         \uucd G 
	 & \leq
	 h(\bar A) + 3 + \uucd G/\bar A
	 \\
	 & \leq h + 3 + f(h-1).
    \end{align*}
    
    \medskip
    
    (2) \textsl{$\{1\} = \bar A$ but there is an element $1\neq a\in
    A$ and a positive integer $m$ such that $a^{t^{m}}=a$:}
    Lemma~\ref{lem:selberg-implication} implies that $m$ is bounded by
    a number $\nu$ that only depends on $h$.  Let $r_{h} =
    \lcm(1,\ldots,\nu)$ and set $A_{0} = A$ and $t_{0}= t^{r_{h}}$.
    Then $G_{0} = A_{0}\rtimes \langle t_{0}\rangle$ is a subgroup of
    $G$ with index $|G:G_{0}| = r_{h}$ and $\bar A_{0}\neq \{1\}$.
    Thus Case~(1) applies to $G_{0}$.  Then
    \begin{align*}
        \uucd G 
	& \leq
	\uugd G
	\\
	& \leq
	r_{h} \uugd G_{0}
	&& 
	\text{(Proposition~\ref{prop:gd-for-finite-index})}
	\\
	& \leq
	r_{h} \max(3,\uucd G_{0})
	&& \text{(Proposition~\ref{prop:lm})}
	\\
	& \leq
	r_{h} (h + 3 +  f(h-1)).
	&&
	\text{(Case (1))}
    \end{align*}
    
    \medskip
    
    (3) \textsl{$\{1\} = \bar A$ and for all $1\neq a\in A$ and all
    $m\neq 0$ we have $a^{t^{m}}\neq a$:} Then
    Proposition~\ref{prop:fluch} applies to $G$ and it follows that
    $\uugd G\leq \uugd A + 1$.  Thus
    \begin{equation*}
        \uucd G \leq \uugd G \leq \uugd A+1 \leq h+3
    \end{equation*}
    where the last inequality is due to Theorems ~4.3 and ~5.13 
    in~\cite{lueckweiermann}.
    
    \medskip
    
    Therefore, if we recursively define $f(h)=r_{h}( h + 3 + f(h-1))$,
    then
    \begin{equation*}
        \uucd G \leq f(h)
    \end{equation*}
     in all three cases.
\end{proof}


\section{Nilpotent-by-abelian groups}

\noindent For any group $G$ we denote its centre by $Z(G)$. 

\begin{lemma}\label{factoroutcentre}
    Let $G$ be a group such that there is a short exact sequence
    \begin{equation*}
	N \mono G \epi Q,
    \end{equation*}
    where $N$ is torsion-free nilpotent and $Q$ is torsion-free.  Then
    $G/Z(N)$ is torsion-free.
\end{lemma}

\begin{proof}
    Since $Z(N)$ is normal in $G$ we have a short exact sequence
    \begin{equation*}
	N/Z(N) \mono G/Z(N) \epi Q
    \end{equation*}
    and it suffices to show that $N/Z(N)$ is torsion-free.  Since
    $Z(N)$ is torsion-free, a Theorem of Mal'cev~\cite{rob}*{5.2.19}
    implies that every upper central factor of $N$ is also
    torsion-free.  Hence, in particular $Z(N/Z(N))=Z_2(N)/Z(N)$ is
    torsion-free.  Applying Mal'cev's result to $N/Z(N)$ yields the
    result as $N/Z(N)$ is nilpotent.
\end{proof}

\begin{theorem}\label{nilp-by-ab} 
    Let $G$ be a group such that there is a short exact sequence
    \begin{equation*}
	N \mono G \epi Q,
    \end{equation*}
    where $N$ is torsion-free nilpotent and $Q$ is torsion-free
    abelian.  Assume that the Hirsch length of $G$ is finite.  Then
    there is a recursively defined integer $g=g(h(N), \linebreak[3] c(N),h(Q))$
    depending only only the Hirsch lengths of $N$ and $Q$ and the
    nilpotency class of $N$ such that
    $$\uucd G \leq g.$$
\end{theorem}

\begin{proof}
    By~\cite{bieribook} and Corollary~\ref{cor:directunions} we may
    assume that $G$ is finitely generated.  $Q$ is finitely generated
    abelian of finite Hirsch length $h(Q)$.  \cite{lueckweiermann}
    implies that $\uucd Q \leq h(Q)+1.$ To apply
    Proposition~\ref{prop:martinez} we need to consider all infinite
    cyclic extensions $H=N \rtimes \langle t \rangle$ and show that
    there is a positive integer $M$ depending only on the Hirsch
    length and the nilpotency class of $N$ such that $\uucd (N \rtimes
    \langle t \rangle) \leq M. $ We prove this by induction on the
    nilpotency class $c$ of $N$.  Let $s_N=\max\{f(1),\ldots
    ,f(h(N))\}$, where $f(n), \, n\in \{1,\ldots  ,h(N)\}$ denotes the
    integer of \ref{ab-by-cyc}.
    
    If $c=1$ we are in the situation of Proposition~\ref{ab-by-cyc}
    and $M\leq f(h(N)) \leq s_N.$

    Now suppose $c>1$ and for all nilpotent groups $N_1$ with
    nilpotency class $<c$ and Hirsch-length $\leq h(N)$, we have an
    integer $M_1,$ depending only on the nilpotency class of $N_1$ and
    the Hirsch length of~$N,$ such that $\uucd(N_1 \rtimes \langle t
    \rangle) \leq M_1$.  $N$ is nilpotent so $Z(N) \neq 1$ and
    $H/Z(N)$ is (torsion-free nilpotent)-by-(infinite cyclic) and
    $c(N/Z(N))<c$, see Lemma \ref{factoroutcentre}.  We also have that
    $h(N/Z(N))\leq h(N).$ Hence, by induction, $\uucd(H/Z(N)) \leq
    M_1.$ Furthermore, we can apply Proposition~\ref{ab-by-cyc} to
    every infinite cyclic extension $T = Z(N) \rtimes \langle t
    \rangle$ of $Z(N)$ which gives $\uucd T \leq f(Z(N))\leq s_N$,
    using that $h(Z(N)) \leq h(N).$ Now apply Martínez-Pérez spectral
    sequence to the short exact sequence
    \begin{equation*}
	Z(N) \mono H \epi H/Z(N)
    \end{equation*}
    to give
    \begin{equation*}
	\uucd H =\uucd (N \rtimes \langle t \rangle) \leq s_N+M_1 = M.
    \end{equation*}
    Since both, $s_N$ and $M_1$ are independent of the choice of
    cyclic extension $H$ of $N$, so is $M$.  Another application of
    the spectral sequence yields
    \begin{equation*}
        \uucd G \leq M +\uucd Q \leq M+h(Q)+1=g.\tag*{\qedhere}
    \end{equation*}
\end{proof}


\section{Proof of the Main Theorem}

The proof is now an easy application of a theorem by Hillman and
Linnell~\cite{hillmanlinnell}.  We shall refer to an alternative proof
of their theorem, see points~(f) and~(g) Wehrfritz \cite{wehrfritz},
whose statement is better suited to our purpose.  For any group~$G$ we
denote by $\tau(G)$ its unique maximal normal locally finite subgroup.

\begin{theorem}\cites{hillmanlinnell,  wehrfritz}
    Let $G$ be an elementary amenable group of finite Hirsch length
    $h$.  Then there is an integer-valued function $j(h)$ of $h$ only
    such that $G$ has characteristic subgroups $\tau(G) \leq N \leq M$
    with $N/\tau(G)$ torsion-free nilpotent, $M/N$ free abelian of
    finite rank and $|G:M|$ at most $j(h).$
\end{theorem}

\begin{proof}[Proof of the Main Theorem]
    Since $G$ has a bound on the orders of the finite subgroups,
    $\tau(G)$ is finite.  An application of
    Proposition~\ref{prop:finkernel} allows us to assume that
    $\tau(G)=\{1\}$ and hence that $G$ is virtually torsion-free.
    Using Proposition~\ref{prop:gd-for-finite-index} we can assume
    that $G$ is torsion-free nilpotent-by-abelian.  Hence we can apply
    Theorem~\ref{nilp-by-ab}.
\end{proof}

\begin{remark}
    To remove the condition that there is a bound on the orders of the
    finite subgroups, one needs to understand virtually cyclic
    extensions of large locally finite groups.  This would allow us to
    apply Mart\'inez-P\'erez' spectral sequence as before.  Since
    $\tau(G)$ is locally finite, every virtually cyclic subgroup is,
    in fact, finite and hence $\underline{\underline E}(\tau(G)) =
    {\underline E}(\tau(G))$ and these are well understood
    \cite{dicksetal}.  In a recent article Degrijse and Petrosyan have
    provided bounds for the dimension of $\underline{\underline E}T$
    for $T$ locally finite-by-virtually cyclic ~\cite{degrijse-11}.
    This implies that every elementary amenable group $G$ admits a
    finite dimensional model for $\underline{\underline E}G.$
\end{remark}

\subsection*{Acknowledgements} 

We would like to thank Ashot Minasyan for numerous helpful
discussions and the referee for carefully reading an earlier version.


\end{document}